\documentclass[12pt]{amsart}
\pdfoutput=1
\usepackage{adjustbox}
\usepackage{parskip}
\usepackage{import}
\usepackage{mathtools}
\usepackage[mathscr]{eucal}
\usepackage{times}
\usepackage{amsmath}
\usepackage{amsthm,amscd}
\usepackage{amssymb}
\usepackage{latexsym}
\usepackage{enumerate}
\usepackage{enumitem}
\usepackage{graphics}
\usepackage{color}
\usepackage[noadjust]{cite}
\usepackage{graphicx}
\usepackage{tikz-cd}
\usepackage[all]{xy}
\usepackage{hyperref}
\hypersetup{
    colorlinks=true,
    urlcolor=black,
    citecolor=blue,   
    linkcolor=blue,
}

\newtheorem{theorem}{Theorem}[section]
\newtheorem{lemma}[theorem]{Lemma}
\newtheorem{proposition}[theorem]{Proposition}
\newtheorem{corollary}[theorem]{Corollary}

\newcommand{\brac}[1]{\left( #1\right)}

\newcommand{\red}[1]{\textcolor{red}{#1}}

\DeclareMathOperator{\vcd}{vcd}

\DeclareMathOperator{\G}{G}
\DeclareMathOperator{\kernel}{Ker}

\DeclareMathOperator{\Image}{Im}

\DeclareMathOperator{\Gal}{Gal}

\DeclareMathOperator{\Spin}{Spin}

\DeclareMathOperator{\N}{N}
\DeclareMathOperator{\M}{M}
\DeclareMathOperator{\HH}{H}

\DeclareMathOperator{\Br}{Br}

\DeclareMathOperator{\charac}{char}

\DeclareMathOperator{\cd}{cd}

\DeclareMathOperator{\GO}{GO}
\DeclareMathOperator{\Oo}{O}
\DeclareMathOperator{\PGO}{PGO}
\DeclareMathOperator{\RPGO}{RPGO}

\DeclareMathOperator{\Sn}{Sn}

\DeclareMathOperator{\R}{R}
\DeclareMathOperator{\W}{W}

\font\brus=wncyr10.240pk scaled 1200 .240pk

\title{The norm principle for type $D_n$ groups over complete discretely valued fields}

\author{Nivedita Bhaskhar}
\address{Department of Mathematics
University of California at Los Angeles \\
Los Angeles, CA 90095-1555}
\email{nbhaskh@math.ucla.edu}
\author{Vladimir Chernousov}
\address{Department of Mathematical Sciences, University of Alberta, Edmonton,
Alberta, Canada T6G 2G1}
\email{vladimir@ualberta.ca}
\author{Alexander Merkurjev}
\address{Department of Mathematics
University of California at Los Angeles \\
Los Angeles, CA 90095-1555}
\email{merkurev@math.ucla.edu}
\thanks{The second author was partially supported by the Canada Research Chairs
Program and an NSERC research grant. The work of the third author has been supported by the NSF grant DMS \#1160206}

\begin{document}
\maketitle

\begin{abstract}
Let $K$ be a complete discretely valued field with residue field $k$ with $\mathrm{char}(k)\neq 2$. Assuming that the norm principle holds for extended Clifford groups $\Omega(q)$ for every even dimensional non-degenerate quadratic form $q$ defined over any finite extension of $k$, we show that it holds for extended Clifford groups $\Omega(Q)$ for every even dimensional non-degenerate quadratic form $Q$ defined over $K$.
\end{abstract}

\section{Introduction}
Let $K$ be a field and $T$, a commutative linear algebraic group defined over $K$. Given $L/K$, a finite separable field extension, one can define the \textit{norm homomorphism} $\N_{L/K} : T(L)\to T(K)$ which sends $t \leadsto \prod_{\gamma} \gamma(t)$ where $\gamma$ runs over cosets of $\Gal\brac{K^{sep}/L}$ in $\Gal\brac{K^{sep}/K}$. The definition of the norm homomorphism can be extended to $K$-\'etale algebras in a similar manner. Note that if $T = \mathbb{G}_m$, then $\N_{L/K} : T(L)\to T(K)$ is precisely the usual norm $\N_{L/K} : L^* \to K^*$.

Now let $G$ be a linear algebraic group defined over $K$ and let $f : G\to T$ be an algebraic group homomorphism defined over $K$. Consider the following diagram: 
\[\begin{tikzcd}
    G(L) \arrow{r}{f(L)}  & T(L) \arrow{d}{{\N_{L/K}}} \\
    G(K) \arrow{r}{f(K)}       & T(K)
    \end{tikzcd}
    \]
    
We say that the \textit{norm principle} holds for $f:G\to T$ over a finite separable field extension (or \'etale algebra) $L/K$ if $\N_{L/K}(\Image f(L)) \subseteq \Image f(K)$. We say that the \textit{norm principle} holds for $f:G\to T$ if for every finite separable field extension (equivalently for every \'etale algebra) $L/K$ , $\N_{L/K}(\Image f(L)) \subseteq \Image f(K)$. 

Suppose further that the commutator subgroup $G'$ of $G$ is defined over $K$. Then every homomorphism $f:G\to T$ factors through the natural homomorphism $\tilde{f} : G\to G/G'$ and it is an easy check that the norm principle for $\tilde{f}$ (over $L/K$) implies the norm principle for $f$ (over $L/K$). We say that the \textit{norm principle} holds for $G$ (over $L/K$) if it holds for $\tilde{f}$ (over $L/K$).

Let $Q$ be a quadratic form over $K$. The classical norm principle of Scharlau which asserts that norms of similarity factors of $Q_L$ are themselves similarity factors of $Q$ can be restated in this context to say that the norm principle holds for the multiplier map $\M : \GO(Q) \to \mathbb{G}_m$. Similarly Knebusch's norm principle which states that norms of spinor norms of $Q_L$ are spinor norms of $Q$ can be reformulated as the norm principle holding for the spinor norm map $\underline{\mu} : \Gamma^+(Q) \to \mathbb{G}_m$. 

Norm principles have been previously studied in  (\cite{Gi93}, \cite{Me96}), especially in conjunction with the rationality or the R-triviality of the algebraic group in question. In (\cite{BM00}), it was established that the norm principle holds in general for all reductive groups of classical type without $D_n$ components. The $D_n$ case was investigated in (\cite{Bh16}) and a scalar obstruction defined up to spinor norms, whose vanishing would imply the norm principle, was given. However, the triviality of this scalar obstruction is far from clear and the question whether the norm principle holds for reductive groups with type $D_n$ components still remains open.

If $K$ is a number field, the norm principle was proved in full generality for all reductive groups by P. Gille (\cite{Gi97}), so the first widely open and very interesting case is when $K$ is the function field $k(C)$  of a curve $C$ defined over a number field $k$ and the group $G$ in question is of classical type with the semisimple part $G'={\rm Spin}(Q)$. As we show in the last section of the paper, 
the validity of the norm principle over $K$ is closely related to the triviality of the kernel of the natural map $\HH^1(K,G')\to \prod_v \HH^1(K_v,G')$ where $v$ runs through a set of discrete valuations of $K$.  Therefore the (traditional) local-global approach leads us first to look in detail over completions $K_v$. 

With this motivation in mind, in this paper, we investigate the $D_n$ case over an arbitrary complete discretely valued field $K$ with residue field $k$ and $\mathrm{char}(k)\neq 2$, restricting ourselves to type $D_n$ groups arising from quadratic forms. In the main result of the paper, we show that if the norm principle holds for such groups defined over all finite extensions of the residue field $k$, then it holds for such groups defined over $K$ (c.f. Theorem \ref{maintheorem}). This yields examples of complete discretely valued fields with residue fields of virtual cohomological dimension $\leq 2$ over which the norm principle holds for the groups under consideration (c.f. Corollary \ref{corollary-examples}). As a further application, we also relate the possible failure of the norm principle to the nontriviality of certain Tate-Shaferevich sets (Section \ref{tateshaferevichsection}).

\subsection*{Notations}
Let $K$ be a field of characteristic not $2$ and $(V, Q)$, a non-degenerate quadratic space of dimension $2n$ over $K$ where $n\in \mathbb{Z}_{>0}$. Let $Z/K$ denote the discriminant extension of $Q$ and let $\psi$ denote the non-trivial $K$-automorphism of $Z$.

Let $\mu$ denote the center of $\Spin(Q)$. Recall that $\mu = \R_{Z/K}\brac{\mu_2}$ when $n$ is even and $\mu=\mu_{4[Z]}:=\kernel\brac{\R_{Z/K}\mu_4\xrightarrow{\mathrm{Norm}}\mu_4}$ when $n$ is odd. Let $\Omega(Q)$ be the \textit{extended Clifford group}\footnote{If $\dim Q = 2$, we set $\Omega(Q)$ to be the commutative group $R_{Z/K}\mathbb{G}_m$, for which the norm principle holds.} of $Q$. If $\dim Q = 2n \geq 4$, this reductive group has center $\R_{Z/K}\mathbb{G}_m$ and is an \textit{envelope} of $\Spin(Q)$ (\cite{BM00}, Ex 4.4). Further it is an extension of the projective group of proper similitudes $\PGO^+(Q)$ by $\R_{Z/K}\mathbb{G}_m$ :
\[1\to \R_{Z/K}\mathbb{G}_m\to \Omega(Q)\xrightarrow{\chi'} \PGO^+(Q)\to 1\]

\section{Reductions}

\subsection{Another formulation of the norm principle}
\label{section-restatingnormprinciple}
Let $1\to J\to G_1\to G_2 \to 1$ be a central $K$-isogeny of reductive groups $G_1, G_2/K$. We recall another formulation of the norm principle for the connecting map $\delta : G_2(-)\to \HH^1(-, J)$ and its relation to the norm principle of an associated map $f : G\to T$. This discussion is taken from (\cite{Me96}, 3.10).

Since $J$ is commutative, one can define norm maps (correstriction) $\N_{L/K}: \HH^1(L, J)\to \HH^1(K, J)$ for finite separable extensions $L/K$. Consider the following diagram: \[\begin{tikzcd}
    G_2(L) \arrow{r}{\delta(L)}  & \HH^1(L, J) \arrow{d}{{\N_{L/K}}} \\
    G_2(K) \arrow{r}{\delta(K)}       & \HH^1(K, J)
    \end{tikzcd}
    \]
    
We say that the \textit{norm principle} holds for $\delta : G_2(-)\to \HH^1(-, J)$ over a finite separable field extension (or \'etale algebra) $L/K$ if $\N_{L/K}(\Image \delta(L)) \subseteq \Image \delta(K)$. We say that the \textit{norm principle} holds for $\delta$ if the norm principle holds for $\delta$ over every finite separable field extension (equivalently over every \'etale algebra) $L/K$. 

Let $T'$ be a quasi-trivial torus\footnote{Product of Weil restrictions of split tori.} containing $J$. The associated map $f : G\to T$ (where $G/K$ is reductive and $T$ is commutative) is determined by the following commutative diagram (Figure 1) with exact rows and columns:

\begin{figure}[hh]
\[\begin{tikzcd}
& 1\arrow{d} & 1\arrow{d} & & \\ 
1\arrow{r} & J\arrow{r}\arrow{d} & G_1 \arrow{r}\arrow{d} & G_2 \arrow{r}\arrow[d, "id"] & 1 \\
1\arrow{r} & T' \arrow{r}\arrow[d, "f'"] & \textcolor{red}{G} \arrow[r, "h"] \arrow[d, "f"] & G_2 \arrow{r} & 1 \\
& \textcolor{red}{T} \arrow[r, "id"]\arrow{d} & \textcolor{red}{T} \arrow{d} & & \\ 
&1 & 1 & & \\ 
\end{tikzcd}\]
\caption{Restating the norm principle}
\end{figure}

\newpage

\begin{lemma}
\label{lemmanormprinciplerestatement} 
Let $L/K$ be a finite separable extension and let $G_1, G_2, J$ and $f: G\to T$ be as above. Then the norm principle holds for $f : G\to T$ over $L/K$ if and only if it holds for $\delta : G_2(-)\to \HH^1(-, J)$ over $L/K$.
\end{lemma}

\begin{proof}
Let $1\to T'\to G\xrightarrow{h}G_2\to 1$ be the exact row as in the diagram above. Since $T'$ is a quasi-trivial torus, this exact sequence induces surjective maps $h(L): G(L)\to G_2(L)$ for all $L/K$.  Let $\delta_1 : T(-)\to \HH^1(-, J)$ be the connecting map of the first exact column in the diagram above. By (\cite{Me96}, Lemma 3.11), the following square is anticommutative.

 \[\begin{tikzcd}
    \textcolor{red}{g \in\ } G(L) \arrow{r}{h(L)}\arrow[d, "f(L)"]  & \textcolor{red}{g_2\in\ }G_2(L)\arrow[d, "\delta(L)"] \\
     \textcolor{red}{t \in\ }T(L) \arrow{r}{\delta_1(L)} & \textcolor{red}{j \in\  }  \HH^1(L, J)
    \end{tikzcd} \]
    
Assume that the norm principle holds for $f$ over $L/K$. Let $g_2\in G_2(L)$ and $\delta(L)(g_2) = j$. Since $h(L)$ is surjective, pick $g\in G(L)$ such that $h(L)(g)= g_2$ and let $t = f(L)(g)\in T(L)$. Thus $\delta_1(L)(t) = j^{-1}$ and $\delta_1(K)(\N_{L/K}(t)) =\N_{L/K}(j^{-1})$. Since the norm principle holds for $f$, $\N_{L/K}(t)= f(K)(\tilde{g})$ for some $\tilde{g}\in G(K)$. Then $\N_{L/K}(j) = \delta(K)[h(K)(\tilde{g})]$. Thus $\N_{L/K}(j)\in \Image \delta(K)$ and hence the norm principle holds for $\delta$ over $L/K$.

Conversely, let the norm principle hold for $\delta$ over $L/K$. Let $g\in G(L)$ and $f(L)(g) = t$. Set $h(L)(g) = g_2\in G_2(L)$ and $\delta(L)(g_2) = j$. Thus $\delta_1(L)(t) = j^{-1}$ and $\delta_1(K)(\N_{L/K}(t)) =\N_{L/K}(j^{-1})$ as before. Since the norm principle holds for $\delta$, $\N_{L/K}(j)= \delta(K)(\tilde{g_2})$ for some $\tilde{g_2}\in G_2(K)$. Since $h(K)$ is surjective, pick $\tilde{g}\in G(K)$ such that $h(K)(\tilde{g}) = \tilde{g_2}$. Then $\N_{L/K}(t) =  f(K)(\tilde{g})f'(K)(t')$ for some $t'\in T'(K)$. Thus $\N_{L/K}(t)\in \Image f(K)$ and hence the norm principle holds for $f$ over $L/K$. \end{proof}

Let $G$ be a reductive group defined over $K$ whose simple components are of classical type and let $T$ be a commutative group defined $K$. Then (Thm 1.1, \cite{BM00}) establishes that if the Dynkin diagram of $G$ does not contain connected components $D_n$ for $n\geq 4$, the norm principle holds for any group homomorphism $G\to T$. We would like to  investigate the norm principle for $G\to T$ in the remaining case when $G$ has simple components of type $D_n$, under the further simplifying assumption that these simple components have simply connected covers $\Spin(Q_i)$ arising from quadratic forms $Q_i$ over $K$.

By following the reductions in (\cite{BM00}), it is easy to see that we need only to check whether the norm principle holds for the  group $\Omega(Q)$, i.e. for the canonical map $\Omega(Q) \to \frac{\Omega(Q)}{\left[\Omega(Q),\Omega(Q)\right]}$ for any non-degenerate even dimensional quadratic form $Q/K$. If $\dim Q  = 2$, as noted before, the norm principle holds for the commutative group $\Omega(Q)$.

\subsection{Maps $S$ and $\alpha$}

We now restate the norm principle for $\Omega(Q)$ when $\dim Q =  2n \geq 4$ in two other equivalent forms, which will be used in the rest of the paper.

Let $T$ denote $\frac{\Omega(Q)}{\left[\Omega(Q),\Omega(Q)\right]}$ and let $T'$ denote the quasi-trivial torus $R_{Z/K}\mathbb{G}_m$ containing $\mu$. Using the fact that the semisimple part of $\Omega(Q)$ is $\Spin(Q)$, Figure 1 above yields the following commutative diagram:

\[\begin{tikzcd}
& 1\arrow{d} & 1\arrow{d} & & \\ 
1\arrow{r} & \mu\arrow{r}\arrow{d} & \Spin(Q) \arrow{r}\arrow{d} & \PGO^+(Q)\arrow{r}\arrow[d, "id"] & 1 \\
1\arrow{r} & R_{Z/K}\mathbb{G}_m \arrow{r}\arrow[d, "f'"] & \textcolor{red}{\Omega(Q)} \arrow[r, "h"] \arrow[d, "f"] &\PGO^+(Q) \arrow{r} & 1 \\
& \textcolor{red}{T} \arrow[r, "id"]\arrow{d} & \textcolor{red}{T} \arrow{d} & & \\ 
&1 & 1 & & \\ 
\end{tikzcd}\]

Let $S : \PGO^+(Q)(-)\to \HH^1(-, \mu)$ denote the connecting map of the first exact row. By Lemma \ref{lemmanormprinciplerestatement}, the norm principle for $\Omega(Q)\to T$ over $L/K$ holds if and only if it holds for $S$ over $L/K$.

For any $L/K$, the short exact sequence $1 \to \mu \to \Spin(Q)\to \PGO^+(Q)\to 1$ gives rise to the long exact sequence 
\[\ldots \to \PGO^+(Q)(L)\xrightarrow{S(L)} \HH^1(L,\mu)\xrightarrow{\alpha(L)} \HH^1(L,\Spin(Q))\to \ldots.\] 
Hence we can deduce that the norm principle holds for $S$ over $L/K$ if and only if the following, which we call the norm principle for $\alpha: \HH^1(-,\mu)\to \HH^1(-,\Spin(Q))$ over $L/K$, holds: 

\textit{For every $u\in \kernel\brac{\HH^1(L,\mu)\xrightarrow{\alpha(L)} \HH^1(L, \Spin(Q))}$, the element $\N_{L/K}(u)$ belongs to $\kernel\brac{\HH^1(K,\mu)\xrightarrow{\alpha(K)} \HH^1(K, \Spin(Q))}$}.

Thus the above discussion gives the following:
\begin{lemma}
\label{lemmaequivnormprinciple} Let $L/K$ be a finite separable field extension. Then the following are equivalent
\begin{enumerate}
\item
The norm principle holds for $\Omega(Q)$ over $L/K$
\item
The norm principle holds for $S : \PGO^+(Q)(-)\to \HH^1(-, \mu)$ over $L/K$
\item
The norm principle holds for $\alpha : \HH^1(-, \mu) \to \HH^1(-, \Spin(Q))$ over $L/K$
\end{enumerate}

\end{lemma}

\subsection{Auxiliary maps $i$ and $j$}
\label{section-auxillarymaps}
We recall the (explicit) definitions of useful auxiliary maps $i : \HH^1(-,\mu_2)\to \HH^1(-,\mu)$ and  $j: \HH^1(-,\mu)\to \HH^1(-,\mu_2)$. 

Recall the following commutative diagram with the two complete rows and columns exact: 

\[\begin{tikzcd}
 & 1 \arrow{d} & 1 \arrow{d} &  & \\
  & \mu_2 \arrow[r, "id"] \arrow{d} & \mu_2 \arrow{d} &  & \\
   1\arrow{r} & \mu \arrow{r}\arrow{d} & \Spin(Q)\arrow{r}\arrow{d} & \PGO^+(Q) \arrow[d, "id"] \arrow{r} & 1\\
      1\arrow{r} & \mu_2 \arrow{d}\arrow{r}& \Oo^+(Q)\arrow{d}\arrow{r} & \PGO^+(Q) \arrow{r} & 1\\
     & 1 & 1  &  &  \end{tikzcd}\]

%
    
For every $L/K$, the long exact sequence of cohomology gives rise to maps $i(L) : \HH^1(L, \mu_2)\to \HH^1(L, \mu)$ and $j(L) : \HH^1(L, \mu)\to \HH^1(L, \mu_2)$ which fit into the diagram 

 \begin{figure}[hh]
\[\begin{tikzcd}
 &  & \HH^1(L, \mu_2) \arrow[r, "id"] \arrow[d, "i(L)"] & \HH^1(L, \mu_2)\arrow[d, "i'(L)"] & \\
 & \PGO^+(Q)(L)\arrow[d, "id"] \arrow[r, "S(L)"] & \HH^1(L, \mu) \arrow[r, "\alpha(L)"] \arrow[d, "j(L)"] & \HH^1(L, \Spin(Q))\arrow{d} & \\
     & \PGO^+(Q)(L) \arrow[r, "\M(L)"]  & \HH^1(L, \mu_2) \arrow{r} & \HH^1(L, \Oo^+(Q)) &   \end{tikzcd}\]
     \label{commdiagram-iandj}
\caption{Auxiliary maps $i$ and $j$}
\end{figure}

Here $\M(L) : \PGO^+(Q)(L)\to \HH^1(L, \mu_2)$ is the multiplier map. The natural map $i'(L) : \HH^1(L,\mu_2)\to \HH^1(L, \Spin(Q))$ in the above diagram will also be used later.

Let $\mathrm{Sn}(L) : \Oo^+(Q)(L)\to \HH^1(L, \mu_2)$ denote the spinor norm map. The maps $i$ and $j$ fit into the following commutative diagram with exact columns (\cite{KMRT98}, Prop 13.33 \& 13.36). 

\begin{figure}[hh]
\[
\begin{tikzcd}
\Oo^+(Q)(L) \arrow{rr}{\mathrm{Sn}(L)} \arrow{d}{\pi(L)} & &   \HH^1(L,\mu_2) \arrow{d}{i(L)} \\
 \PGO^+(Q)(L) \arrow{rr}{S(L)} \arrow{d}{\M(L)} & & \HH^1\left(L, \mu\right)\arrow{d}{j(L)}\\
  \HH^1(L,\mu_2)&  =  &  \HH^1(L,\mu_2)
\end{tikzcd}\]
\label{commdiagram-spinornorms}
\caption{Relation to spinor norms}
\end{figure}

The explicit descriptions of $i$ and $j$ depend on the parity of $n=\dim(Q)/2$. 
\begin{enumerate}
\item[-]
If $n$ is even, $\HH^1(K, \mu) = \frac{Z^*}{Z^{*2}}$. Then $i(K) : \frac{K^*}{K^{*2}} \to \frac{Z^*}{Z^{*2}}$ is the inclusion map and $j(K) : \frac{Z^*}{Z^{*2}}\to \frac{K^*}{K^{*2}}$ sends $[z]\leadsto [\N_{Z/K}(z)]$ for $z\in Z^*$.
\item[-] 
If $n$ is odd, $\HH^1(K,\mu) = \frac{U(K)}{U_{0}(K)}$ where $U\subset \mathbb{G}_m \times \R_{Z/K}\mathbb{G}_m$ is the subgroup defined by $U(K) = \{(f,z)\in K^*\times Z^*| f^4 = \N_{Z/K}(z)\}$ and $U_0\subset U \subset \mathbb{G}_m \times \R_{Z/K}\mathbb{G}_m$ is the algebraic subgroup defined by $U_0(K) = \{(\N_{Z/K}(z),z^4)\in K^*\times Z^*| z\in Z^*\}$.

Then $i(K): \frac{K^*}{K^{*2}} \to \frac{U(K)}{U_0(K)}$ is the map sending $f K^{*2}\leadsto [f, f^2]$ for $f\in K^*$. Finally $j(K) : \frac{U(K)}{U_0(K)} \to \frac{K^*}{K^{*2}}$ sends $[f,z]\leadsto \N_{Z/K}(z_0)K^{*2}$ where $z_0\in Z^*$ is such that $z_0\psi(z_0)^{-1} = f^{-2}z$. 
\end{enumerate}

We end this section with one more useful reduction, which is a direct consequence of Knebusch's norm principle. 

\begin{lemma} Let $L/K$ be a finite separable extension and $u\in \kernel(\alpha(L))$. If $j(L)(u) = 1\in \HH^1(L, \mu_2)$, then $\N_{L/K}(u)\in  \kernel(\alpha(K))$. 
\label{lemmaspinornorm}
\end{lemma}

\begin{proof} Since $u\in \kernel\brac{\alpha(L)}$, there exists $[g]\in \PGO^+(Q)(L)$ such that $S(L)([g]) = u$. Since $j(L)(u) = 1$, by Figure 3 above, we see that $u$ is the image of a spinor norm of $Q_L$, i.e $u=i(L)(\Sn(h))$ for some $h\in \Oo^+(Q)(L)$. By Knebusch's norm principle, $\N_{L/K}(\Sn(h)) = \Sn(\tilde{h})$ for some $\tilde{h}\in \Oo^+(Q)(K)$. Hence $\N_{L/K}(u) = i(K)\brac{\Sn(\tilde{h})} = S(K)(\pi(\tilde{h}))$. Therefore $\alpha(K)(\N_{L/K}(u)) = 1$. \end{proof}

\subsection{A square diagram}
\label{sectionsquare}
Using the auxiliary maps $i$ and $j$ just defined, we construct a square diagram (SQ) by piecing together the two classical norm principles of Scharlau and Knebusch. It is shown that the norm principle holds for $\Omega(Q)$ precisely when the square is commutative for all finite separable extensions. This reformulation yields further reductions for the norm principle question.

Using Figure 2 in Section \ref{section-auxillarymaps}, for each finite separable extension $L/K$, define the subgroup $H(L) \subseteq \HH^1(L, \mu)$ as follows:

\begin{align*}
H(L) &:= \{x\in \HH^1(L, \mu) \ |\ j(L)(x) \in \Image\brac{\M(L)}  \}\\
& = \{x\in \HH^1(L, \mu) \ |\ \alpha(L)(x) \in \Image\brac{i'(L)}  \}.
\end{align*}

Recall that $i'(L) : \HH^1(L, \mu_2)\to \HH^1(L, \Spin(Q))$ induces the group homomorphism $i''(L) : L^*/L^{*2}\to L^*/\Sn(Q_L) \subseteq \HH^1(L, \Spin(Q))$ sending $fL^{*2}\leadsto [f]$ for each $f\in L^*$. Thus $\alpha(L)$ induces the following map \[\tilde{\alpha}(L) : H(L)\to L^*/\Sn(Q_L).\]
 
 \begin{lemma}
$\tilde{\alpha}(L) : H(L)\to L^*/\Sn(Q_L)$ is a group homomorphism.
\label{lemmaalphaisagrouphom}
\end{lemma}

\begin{proof}
For $i=1,2$, let $z_i\in H(L)$ with $\tilde{\alpha}(L)(z_i) = x_i \in L^*/\Sn(Q_L)$. By the definition of $H(L)$ and a diagram chase of Figure 2 of Section \ref{section-auxillarymaps}, there exist $[g_i]\in \PGO^+(L)$ such that $S([g_i^{-1}])z_i = i(L)(y_i)$ for $y_i\in \HH^1(L,\mu_2)$. Thus $\tilde{\alpha}(L)(S([g_i^{-1}])z_i )=i'(L)(y_i)=i''(L)(y_i)$. Since $i''(L)$ is a group homomorphism and $H(L)$ is abelian, we see that 
\begin{align*}
\tilde{\alpha}(L)(S([g_1^{-1}])z_1)\tilde{\alpha}(L)(S([g_2^{-1}])z_2) &= i''(L)(y_1)i''(L)(y_2) \\
& = i''(L)(y_1y_2) \\
& =  \tilde{\alpha}(L)(S([g_1^{-1}])z_1S([g_2^{-1}])z_2)\\
& = \tilde{\alpha}(L)(S([g_2^{-1}g_1^{-1}])z_1z_2 ).
\end{align*}

Since $\mu = Z(\Spin(Q))$, we have $\tilde{\alpha}(L)(S([g_i^{-1}]z_i) = \tilde{\alpha}(L)(z_i)$ and $\tilde{\alpha}(L)(S([g_2^{-1}g_1^{-1}]z_1z_2) = \tilde{\alpha}(L)(z_1z_2)$ (\cite{KMRT98}, Corollary 28.4, Pg 386), we conclude $\tilde{\alpha}(L)$ is a group homomorphism. \end{proof}

From the definition of $H(L)$, it follows that $\Image\brac{i(L)}\subseteq H(L)$. Further by a chase of the following square which is part of Figure 2, we see that $\tilde{\alpha}(L)(i(L)[fL^{*2}])= [f]$ for each $f\in L^*$.

\begin{figure}[hh]
\[\begin{tikzcd}
\red{fL^{*2}\in }\HH^1(L, \mu_2) \arrow[r, "id"] \arrow[d, "i(L)"] & \red{fL^{*2}\in }\HH^1(L, \mu_2)\arrow[d, "i'(L)"]  \\
  \HH^1(L, \mu) \arrow[r, "\alpha(L)"] &   \red{[f]\in } \HH^1(L, \Spin(Q))
  \end{tikzcd}\]
\end{figure}

Similarly it is immediate to see that $\Image\brac{S(L)}=\kernel\brac{\alpha(L)}\subseteq H(L)$ and further is exactly $\kernel\brac{\tilde{\alpha}(L)}$. 

Scharlau's norm principle implies that the norm map $\N_{L/K} : \HH^1(L,\mu)\to \HH^1(K,\mu)$ induces the map $\N_{L/K} : H(L)\to H(K)$. Similarly Knebusch's norm principle implies that the norm map $\N_{L/K} : L^*\to K^*$ induces the map $\N_{L/K} : L^*/\Sn(Q_L)\to K^*/\Sn(Q)$. Thus, the following square diagram (labelled SQ for $L/K$) is defined:

\[\begin{tikzcd}
H(L) \arrow[r, "\tilde{\alpha}(L)"] \arrow[d, "\N_{L/K} (\mathrm{Scharlau})"] & L^*/\Sn(Q_L) \arrow[d, "\N_{L/K} (\mathrm{Knebusch})"] \\
H(K) \arrow[r, "\tilde{\alpha}(K)"] & K^*/\Sn(Q)
\end{tikzcd}\]

\begin{theorem}
\label{thmnormprinciplesquare} Let $L/K$ be a finite separable field extension. Then the following are equivalent:
\begin{enumerate}
\item
The norm principle holds for $\Omega(Q)$ over $L/K$.
\item
The square diagram (SQ for $L/K$) commutes.
\end{enumerate}
\end{theorem}

\begin{proof}

\textbf{(1) $\implies$ (2)}: By Lemma \ref{lemmaequivnormprinciple}, the norm principle holds for $\alpha$ over $L/K$. Let $z\in H(L)$ and $\tilde{\alpha}(L)(z) = [x]\in L^*/\Sn(Q_L)$ for $x\in L^*$. As observed above, $x\in H(L)$ and $\tilde{\alpha}(L)(x)=[x]\in L^*/\Sn(Q_L)$. Since $\tilde{\alpha}(L)$ is a group homomorphism, $zx^{-1} \in \kernel(\tilde{\alpha}(L)) = \kernel(\alpha(L))$. Since we assume that the norm principle holds for $\alpha$ over $L/K$, $\alpha(K)(\N_{L/K}(zx^{-1}))=1$. As $\N_{L/K}(z), \N_{L/K}(x), \N_{L/K}(zx^{-1})\in H(K)$ and $\alpha(K)$ induces $\tilde{\alpha}(K)$, $\tilde{\alpha}(K)(\N_{L/K}(z)) = \tilde{\alpha}(K)(\N_{L/K}(x)) = [\N_{L/K}(x)]\in K^*/\Sn(Q)$. Hence (SQ for $L/K$) commutes.
\[\begin{tikzcd}
\red{z\in \ }H(L) \arrow[r, "\tilde{\alpha}(L)"] \arrow[d, "\N_{L/K}"] & \red{[x]\in \ }L^*/\Sn(Q_L) \arrow[d, "\N_{L/K}"] \\
\red{\N_{L/K}(z)\in \ }H(K) \arrow[r, "\tilde{\alpha}(K)"] & \red{[\N_{L/K}(x)]\in \ }K^*/\Sn(Q) 
\end{tikzcd}\]

\textbf{(2) $\implies$ (1)}: Let $u\in \kernel(\alpha(L)) = \kernel(\tilde{\alpha}(L))$. By the commutativity of the square (SQ for $L/K$), we have $\tilde{\alpha}(K)\brac{\N_{L/K}(u)}=1$. Hence the norm principle holds for $\alpha$ over $L/K$ and by Lemma \ref{lemmaequivnormprinciple}, it holds for $\Omega(Q)$ over $L/K$. \end{proof}

If $Q/K$ is isotropic, then $\Sn(Q) = K^*$ and hence the square (SQ for $L/K$) commutes for all $L/K$ yielding the following:

\begin{corollary} If $Q/K$ is isotropic, then the norm principle holds for $\Omega(Q)$.
\label{corollaryisotropicOK}
\end{corollary}

\subsubsection{Reduction to quadratic extensions}

We now show that to prove the norm principle for $\Omega(Q)$, it suffices to consider separable quadratic extensions. More precisely, we prove the following:

\begin{theorem}
\label{theoremreductiontoquadratic} 
\label{theorem-quadraticextensionreduction}
Let $Q/K$ be a quadratic form of even dimension. Suppose that for every finite separable field extension $M/K$,  the norm principle holds for the $M$-group $\Omega(Q)_M$ over every separable quadratic field extension $M'/M$. Then the norm principle holds for $\Omega(Q)$.
\end{theorem}

\begin{proof} Let $L/K$ be any finite separable field extension. By Theorem \ref{thmnormprinciplesquare}, it suffices to show that the square (SQ for $L/K$) commutes. 

There exists a separable field extension $M/K$ with $[M:K]=2m+1$ for\footnote{For instance, take $M$ to be the fixed field of a $2$-sylow subgroup of $\Gal(N/L)$ where $N$ is the Galois closure of $L$ over $K$.} some $m\geq 0$ such that $LM:= L\otimes_K M \simeq \prod_{i=1}^{r}M_i$ where each $[M_i: M]=2^{m_i}$ with $M_i/M$ separable field extensions filtered by quadratic extensions.

%
%
%

By assumption and Lemma \ref{lemmaequivnormprinciple}, the norm principle holds for the map $\alpha$ over each $M_i/M$ and hence over the \'etale algebra $LM/M$. Hence, by Theorem \ref{thmnormprinciplesquare}, the square (SQ for $LM/M$) commutes. 

Note that the natural map $K^*/\Sn(Q)\to M^*/\Sn(Q_M)$ is injective. This is because if $f\in K^*$ becomes a spinor norm from $Q_M$, then by Knesbusch's norm principle, $\N_{M/K}(f) = f^{2m+1}$ is a spinor norm from $Q$ and hence $[f] = 1 \in K^*/\Sn(Q)$ to begin with.

Look at the following cuboid which has commutative front vertical face (SQ for $LM/M$) as well as commutative side, top and bottom faces. A diagram chase and the injectivity of the map $K^*/\Sn(Q)\to M^*/\Sn(Q_M)$ shows that the back vertical face (SQ for $L/K$) commutes.
\[\begin{tikzcd}
H(L)\arrow[ddr, dashed] \arrow[rr, "\tilde{\alpha}(L)"]\arrow[d, "\N"] & & L^*/\Sn(Q_L) \arrow[d, "\N"] \arrow[ddr, dashed]  & &\\
H(K)\arrow[ddr, dashed]\arrow[rr, "\tilde{\alpha}(K)"] & & K^*/\Sn(Q) \arrow[ddr, dashed, hookrightarrow, crossing over] & & & \\
& H(LM)\arrow[rr, "\tilde{\alpha}(LM)"]\arrow[d, "\N"] & & (LM)^*/\Sn(Q_{LM}) \arrow[d, "\N"]  & \\
 & H(M)\arrow[rr, "\tilde{\alpha}(M)"] & & M^*/\Sn(Q_M) & \\
\end{tikzcd}\] \end{proof}

\section{An inductive approach}
In this section, we outline a possible inductive approach to the norm principle question.

Let $Q/K$ be a quadratic form of even dimension as before and let $L/K$ be a finite separable field extension. Let $u\in \kernel(\alpha(L))$. We would like to show $\N_{L/K}(u)\in \kernel(\alpha(K))$. Set $j(L)(u) = [\lambda]\in \HH^1(L, \mu_2)$ for some $\lambda\in L^*$. It follows from Figure 2 in Section \ref{section-auxillarymaps} that $Q_L \simeq \lambda Q_L$.

\textit{Suppose that there exist even dimensional quadratic forms $f_u, g_u$ defined over $K$ such that $Q \simeq f_u \perp g_u$ and $\brac{f_u}_L \simeq \lambda \brac{f_u}_L$.} Note that this immediately implies $\brac{g_u}_L\simeq \lambda \brac{g_u}_L$. Let $R_u := \Oo^+(f_u)\times \Oo^+(g_u)\subset \Oo^+(Q)$. Define an intermediate new group $\tilde{R}_u$ to be the preimage of $R_u$ under the canonical homomorphism $\Spin(Q)\to \Oo^+(Q)$.

We have the following diagram with exact rows:
\[\begin{tikzcd}
1\arrow{r} & \mu_2\arrow{r}\arrow[d, "id"] & \mu \arrow{r}\arrow[hookrightarrow]{d} &  \mu_2 \arrow{r}\arrow[hookrightarrow]{d} & 1 \\
 1\arrow{r} & \mu_2\arrow{r}\arrow[d, "id"] & \tilde{R}_u \arrow{r}\arrow[hookrightarrow]{d} &  R_u \arrow{r}\arrow[hookrightarrow]{d} & 1 \\
 1\arrow{r} & \mu_2\arrow{r}& \Spin(Q) \arrow{r} &  \Oo^+(Q) \arrow{r}& 1
\end{tikzcd}\]

Let $\gamma_u : \HH^1(-,\mu)\to \HH^1(-,\tilde{R}_u)$ be the induced map from the inclusion $\mu\hookrightarrow\tilde{R}_u$.

\begin{lemma}
\label{lemmainductiveapproach}
Let $L/K$ be a finite separable field extension and $u\in \kernel(\alpha(L))$. Assume that there exist even dimensional quadratic forms $f_u, g_u/K$ such that $Q\simeq f_u\perp g_u$ and $j(L)(u)=[\lambda]$ for $\lambda\in L^*$ with $\brac{f_u}_L \simeq \lambda \brac{f_u}_L$ and $\brac{g_u}_L \simeq \lambda \brac{g_u}_L$. If $\N_{L/K}\brac{\kernel\brac{\gamma_u(L)}}\subset \kernel\brac{\gamma_u(K)}$, then $\N_{L/K}(u)\in \kernel\brac{\alpha(K)}$.
\end{lemma}

\begin{proof} Let $\gamma_u(L)(u)=v\in \HH^1(L, \tilde{R}_u)$. Since $\brac{f_u}_L \simeq \lambda \brac{f_u}_L$ and $\brac{g_u}_L \simeq \lambda \brac{g_u}_L$, $[\lambda]$ goes to $1$ in $\HH^1(L, R_u)$. Hence $v$ goes to $1$ in $\HH^1(L, R_u)$ and therefore there exists $a\in L^*$ such that $[a]\in \HH^1(L, \mu_2)$ goes to $v$.

\adjustbox{scale=1, center}{
$\footnotesize{\begin{tikzcd}
& \HH^1(L, \mu_2) \arrow[rr, "i(L)"] \arrow[dl, "id"] \arrow[dd, dashrightarrow] & & \red{u\in}\HH^1(L, \mu) \arrow[rr, "j(L)"] \arrow[dl, "\gamma_u(L)"] \arrow[dd, dashrightarrow, "\alpha(L)"] & &  \red{[\lambda]\in}\HH^1(L, \mu_2) \arrow{dl} \arrow[dd, dashrightarrow] \\
\HH^1(L, \mu_2) \arrow{rr} \arrow[dr, "id"] &  & \red{v\in} \HH^1(L, \tilde{R}_u)   \arrow{rr} \arrow{dr}  &  & \red{1\in} \HH^1(L, R_u)  \arrow{dr} &   \\
& \HH^1(L, \mu_2) \arrow{rr}  & & \red{1\in}\HH^1(L, \Spin(Q)) \arrow{rr} & &  \HH^1(L, \Oo^+(Q))  \\
\end{tikzcd}}$
}

As $\mu\subseteq Z(\tilde{R}_u)$, we have $i([a])^{-1}u$ is in the image of $\brac{\tilde{R}_u/\mu}(L)\to \HH^1(L, \mu)$ (\cite{KMRT98}, Corollary 28.4, Pg 386). Hence $\gamma_u(L)\brac{i([a])^{-1}u}=1$ and therefore $\alpha(L)(i([a])^{-1}u)=1\in \HH^1(L, \Spin(Q))$. Since $\alpha(L)(u)=1$, we see that $\alpha(L)(i[a])=1$, i.e. $a$ is a spinor norm of $Q_L$. By Knebusch's norm principle, $\N_{L/K}(a)$ is a spinor norm of $Q$ and hence $\alpha(K)\brac{i[N_{L/K}(a)]}=1$.

By assumption on $\gamma_u$, we have $\gamma_u(K)\brac{N_{L/K}\brac{i([a])^{-1}u}}=1\in \HH^1(K, \tilde{R}_u)$ and hence this element dies in $\HH^1(L, \Spin(Q))$ also. That is, $\alpha(K)\brac{\N_{L/K}\brac{i([a])^{-1}u}}=1$. This implies $\alpha(K)(\N_{L/K}(u))=1$. \end{proof}

This leads to the following:

\begin{theorem}
\label{theoreminductiveapproach}
Let $L/K$ be a finite separable field extension and $u\in \kernel(\alpha(L))$. Assume that there exist even dimensional quadratic forms $f_u, g_u/K$ such that $Q\simeq f_u\perp g_u$ and $j(L)(u)=[\lambda]$ for $\lambda\in L^*$ with $\brac{f_u}_L \simeq \lambda \brac{f_u}_L$ and $\brac{g_u}_L \simeq \lambda\brac{g_u}_L$. If the norm principle holds for $\Omega(f_u)$ and $\Omega(g_u)$ over $L/K$, then $\N_{L/K}(u)\in \kernel\brac{\alpha(K)}$. \end{theorem}

\begin{proof} Recall the exact sequence of algebraic $K$-groups $1\to \mu\to \tilde{R}_u \to \tilde{R}_u/\mu\to 1$. By Lemma \ref{lemmainductiveapproach}, it suffices to check that $\N_{L/K}\brac{\kernel\brac{\gamma_u(L)}}\subset \kernel\brac{\gamma_u(K)}$, which is equivalent to verifying that the norm principle holds for $ \tilde{R}_u/\mu(-)\to \HH^1(-, \mu)$ over $L/K$ (c.f. Lemma \ref{lemmaequivnormprinciple}). As in Section \ref{section-restatingnormprinciple}, construct $f: G\to T$ given by the following commutative diagram with exact rows and columns. 

\[\begin{tikzcd}
& 1\arrow{d} & 1\arrow{d} & & \\ 
1\arrow{r} & \mu\arrow{r}\arrow{d} & \tilde{R}_u \arrow{r}\arrow{d} & \tilde{R}_u/\mu \arrow{r}\arrow[d, "id"] & 1 \\
1\arrow{r} & R_{Z/K}\mathbb{G}_m \arrow{r}\arrow[d, "f'"] & \textcolor{red}{G} \arrow[r, "h"] \arrow[d, "f"] & \tilde{R}_u/\mu  \arrow{r} & 1 \\
& \textcolor{red}{T} \arrow[r, "id"]\arrow{d} & \textcolor{red}{T} \arrow{d} & & \\ 
&1 & 1 & & \\ 
\end{tikzcd}\]

By Lemma \ref{lemmanormprinciplerestatement}, it suffices to check the norm principle for $f: G\to T$ over $L/K$ or more generally, the norm principle for $G$ over $L/K$. Then (\cite{BM00}, Proposition 5.2) along with the hypothesis that the norm principle holds for $\Omega(f_u)$ and $\Omega(g_u)$ concludes the proof. \end{proof}

\section{R-equivalence classes of $\PGO^+(Q)(K)$}
Let $G$ be a linear algebraic group defined over $K$ and $L/K$, a field extension. Then $x, y\in G(L)$ are said to be \textit{$R$-equivalent} if there exists an $L$-rational map $f : \mathbb{A}^1_{L}\dashrightarrow G$ defined at $0$ and $1$ sending $0 \leadsto x$ and $1 \leadsto y$. This defines an equivalence relation on $G(L)$ (\cite{Gi97}, Section II.1). Let $RG(L)$ denote the normal subgroup of elements in $G(L)$ which are $R$-equivalent to the identity $e_G$. We denote the quotient group $G(L)/RG(L)$ by $G(L)/R$. This is the group of $R$-equivalence classes introduced by Manin for the $L$-points of the variety underlying group $G$.

The norm principles  in  (\cite{Gi93}, \cite{Me96}) are stated in general for $R$-trivial elements (i.e, elements $R$-equivalent to $e_G$). 

\begin{theorem}[Gille, \cite{Gi93}]
\label{gillenormprinciple}
Let $1\to \mu\to \tilde{G}\to G \to 1$ be an isogeny of semisimple algebraic groups over $K$ and $\N_{L/K} : \HH^1\brac{L, \mu}\to \HH^1\brac{K, \mu}$ be the induced norm map for a field extension $L/K$. Let $RG(L)$ (resp. $RG(K)$) denote the elements of $G(L)$ (resp $G(K)$) which are $R$-equivalent to the identity.

\[\begin{tikzcd}
    RG(L) \arrow{r}{\delta(L)}  & \HH^1\brac{L, \mu} \arrow{d}{{\N_{L/K}}} \\
    RG(K) \arrow{r}{\delta(K)}       & \HH^1\brac{K, \mu}
    \end{tikzcd}
    \]

 Then $\N_{L/K}\brac{\Image \delta(L) : RG(L) \to \HH^1\brac{L, \mu}} \subseteq \brac{\Image \delta(K) : RG(K) \to \HH^1\brac{K, \mu}}$.
\end{theorem}

A similar statement holds for the norm principles of morphisms $f : G\to T$ where $G$ is a reductive linear algebraic group (\cite{Me96}, Thm 3.9).


In \cite{Me96(2)}, the group of $R$-equivalence classes was computed for adjoint semisimple classical groups. These computations applied to the adjoint group $\PGO^+(Q)$ translates to a natural isomorphism $\PGO^+(Q)(K)/R  \ \simeq  \ \G(Q)/\brac{K^{*2} \  \mathrm{Hyp}(Q)}$ where,

\begin{enumerate}
\item[-] $Q/K$ is a non-degenerate form of dim $2n$,
 \item[-] $\G(Q)$ is the \textit{group of similarities} $\{\lambda \in K^* \ | \ \lambda Q \simeq Q\}$,
 \item[-] $\mathrm{Hyp}(Q)$ is the subgroup generated by $\langle \N_{E/K}\brac{E^*}\ |\  Q_E \simeq \mathbb{H}^n  \rangle$ where $E$ runs over finite extensions of $K$.
 \end{enumerate}
 
The following lemma identifies some quadratic forms which give rise to adjoint groups of type $D_n$ with trivial group of $R$-equivalence classes over the base field.

\begin{lemma}
\label{Rtrivialquadform} Let $K$ be a complete discretely valued field with ring of integers $\mathcal{O}_K$ and residue field $k$ with $\mathrm{char}(k)\neq 2$. Let $\tau = \langle a_1, a_2, \ldots, a_r \rangle$ where each $a_i\in \mathcal{O}_K^*$. Let $\pi\in K^*$ be a parameter of $K$ and set $ \xi = \tau \otimes \langle 1, \pi \rangle$. Then $\PGO^+(\xi)(K)/R$ is trivial.
\end{lemma}

\begin{proof} 
Let $\theta\in \G(\xi)$. Thus up to squares in $K^*$, $\theta = x\pi^{\epsilon}$ for some $x\in \mathcal{O}_K^*$ and $\epsilon \in \{0,1\}$. Since $\PGO^+(\xi )(K)/R  \ \simeq  \ \G(\xi )/\brac{K^{*2} \  \mathrm{Hyp}(\xi )}$, we would like to show $x\pi^{\epsilon}\in \mathrm{Hyp}(\xi )K^{*2}$. 

The totally ramified extension $L' = K(\sqrt{-\pi})$ splits $\xi $ and $\pi = N_{L'/K}(\sqrt{-\pi})$. Thus $\pi\in \mathrm{Hyp}(\xi )$. Thus we can assume $\theta = x \in \G(\xi )$ for $x\in \mathcal{O}_K^*$.

Let $\overline{x}$ denote the image of $x$ in $k$ and $\overline{\tau}$, the image of $\tau$ in $W(k)$. Recall the second residue homomorphism $\delta_{2,\pi} : \W(K)\to \W(k)$ with respect to the parameter $\pi$. Then we have $[\overline{\tau}] = \delta_{2,\pi}(\xi)=\delta_{2,\pi}\brac{x \pi^{\epsilon}\xi} = \delta_{2,\pi}\brac{x\xi}  = [\overline{x}\overline{\tau}] \in W(k)$. This shows that $[\overline{\tau}\otimes {\langle 1, -\overline{x} \rangle}_k] = 0$ in $W(k)$.

Look at $L'' = K(\sqrt{-x\pi})$ which is a complete discretely valued field with residue field $k$. Recall the first residue homomorphism $\delta_{1,y} : \W(L'')\to \W(k)$ with respect to some  parameter $y$ of $L''$. Note that $\xi_{L''} \simeq \tau_{L''}\otimes {\langle 1, -x\rangle}_{L''}$ and $\delta_{1,y}(\xi_{L''})= [\overline{\tau}\otimes \langle 1, -\overline{x} \rangle_k] = 0\in W(k)$. Thus by Hensel's Lemma, $[\xi_{L''}] = 0 \in W(L'')$.

Finally as $x\pi = N_{L''/K}(\sqrt{-x\pi})$, $x\pi \in \mathrm{Hyp}(\xi)$. Since $\mathrm{Hyp}(\xi)$ is the subgroup generated by norms from finite extensions which split $S$, $x \in \mathrm{Hyp}(\xi)$, which concludes the proof. \end{proof}

%
%

\section{Over complete discretely valued fields}
In this section, we work over a complete discretely valued field $K$. We fix the convention of letting  $\overline{\star}$ denote the image of $\star$ in the residue field for $\star$ defined over the ring of integers of a complete discretely valued field. We also let $\mathcal{O}_X$ denote the ring of integers of a complete discretely valued field $X$. We now state the main result of this paper.

\begin{theorem}
\label{maintheorem}
Let $K$ be a complete discretely valued field with residue field $k$ with $\mathrm{char}(k)\neq 2$. Assume that the norm principle holds for $\Omega(q)$ for every non-degenerate quadratic form $q$ of even dimension defined over any finite extension of $k$. Then the norm principle holds for $\Omega(Q)$ for every non-degenerate quadratic form $Q$ of even dimension over $K$.
\end{theorem}



By Theorem \ref{theorem-quadraticextensionreduction}, it suffices to show that the norm principle holds for $\Omega(Q)$ over separable quadratic extensions $L/K$. Fix a uniformizing parameter $t$ of $K$. Write the non-degenerate quadratic form $Q/K$ in the form $q\perp tp$ with $q \simeq \langle a_1, a_2, \ldots, a_{\dim q}\rangle$ and $p \simeq \langle b_1, b_2, \ldots, b_{\dim p} \rangle$ for $a_i, b_j\in \mathcal{O}_K^*$. Since $\dim Q = 2n$, $\dim q$ and $\dim p$ have the same parity. 

 Note that $L/K$ is a complete discretely valued field. Let $\ell/k$ denote its residue field. Then one of the following holds:
\begin{enumerate}
\item[-] 
$L/K$ is an unramified extension and $\theta:= t$ is a parameter of $L$. 
\item[-] $L/K$ is totally ramified and $\ell = k$. Further $L\simeq K(\sqrt{ct})$ for some $c\in \mathcal{O}_K^*$ and $\theta:=\sqrt{ct}$ is a parameter of $L$.
\end{enumerate}

 By Lemma \ref{lemmaequivnormprinciple}, it suffices to verify the norm principle for $\alpha$ over $L/K$. Let $u\in \kernel\brac{\alpha(L)} \subseteq H(L)$. Set $j(L)(u) = [\lambda] \in \HH^1(L,\mu_2)\simeq L^*/L^{*2}$ for some representative $\lambda\in L^*$. We would like to show that  ${\alpha}(K)\brac{\N_{L/K}(u)} = 1\in \HH^1(K, \Spin(Q))$. By Corollary \ref{corollaryisotropicOK}, we can assume $q$ and $p$ are anisotropic over $K$.

\subsection{Lemmata}
We begin by reducing to the case when $\lambda$ is a unit in $L$.

\begin{lemma}
Up to squares, $\lambda$ can be assumed to be in $\mathcal{O}_L^*$.
\label{lemmalambdaisaunit}
\end{lemma}

\begin{proof}
Without loss of generality, we can assume $\lambda\in \mathcal{O}_L^*$ or $\lambda = \tilde{\lambda}\theta$ for $\tilde{\lambda}\in \mathcal{O}_L^*$ where $\theta$ is a parameter of $L$. 

If $L/K$ is unramified and $\lambda = \tilde{\lambda}\theta$ where $\theta = t$, using $\lambda Q_L  \simeq Q_L$ and the second residue map $\delta_{2,t} : W(L)\to W(\ell)$ for the parameter $t$ of $L$, we see that \[[\overline{p}_{\ell}] = \delta_{2,t}(Q_L)=\delta_{2, t}\brac{\lambda Q_L} = [\overline{\tilde{\lambda}}\overline{q}_{\ell}] \in W(\ell).\] Thus $Q_L\simeq  q_L \otimes {\langle 1, \tilde{\lambda} t\rangle}_L$. By Lemma \ref{Rtrivialquadform}, $\PGO^+(Q)(L)/R =\{1\}$ and hence $\RPGO^+(Q)(L) = \PGO^+(Q)(L)$. Then Theorem \ref{gillenormprinciple} implies that the norm principle holds for our required map $S : \PGO^+(Q)\to \HH^1(-,\mu)$ over the extension $L/K$ and hence for $\Omega(Q)$ over $L/K$ (Lemma \ref{lemmaequivnormprinciple}). 

If $L\simeq K(\sqrt{ct})$ for some $c\in \mathcal{O}_K^*$ and $\lambda = \tilde{\lambda}\theta$ where $\theta:=\sqrt{ct}$ and $Q_L\simeq q_L \perp cp_L$, using $\lambda Q_L  \simeq Q_L$ and the second residue map $\delta_{2,\theta} : W(L)\to W(k)$ for the parameter $\theta$ of $L$, we see that \[0 = \delta_{2,\theta}(Q_L)=\delta_{2, \theta}\brac{\lambda Q_L} = [\overline{\tilde{\lambda}}\overline{q} \perp \overline{\tilde{\lambda}}\overline{cp}] \in W(k).\] Thus $[\overline{q}\perp \overline{c}\overline{p}] = 0\in W(k)$ and hence $[Q_L] = 0 \in W(L)$. This implies $\PGO^+(Q)(L)/R =\{1\}$ and hence $\RPGO^+(Q)(L) = \PGO^+(Q)(L)$. Then Theorem \ref{gillenormprinciple} and Lemma \ref{lemmaequivnormprinciple} imply that the norm principle holds for $\Omega(Q)$ over $L/K$. \end{proof}

Next, we describe the shape of the element $u\in \kernel \brac{\alpha(L)}$ under consideration when $Q_L$ is unramified. We construct a related element $u'\in \HH^1(\mathcal{O}_L, \mu)$ which in fact also lives in $\kernel \brac{\alpha(L)}$. To do so, we make use of the explicit description of the maps $i(L)$ and $j(L)$ given in Section \ref{section-auxillarymaps}. 

\begin{lemma} Let $L$ be a complete discretely valued field with rings of integers $\mathcal{O}_L$, parameter $\theta$ and residue field $\ell$. Let $Q' =  \langle x_1, x_2, \ldots, x_{2r}\rangle$ be a quadratic form over $L$ where each $x_i\in \mathcal{O}_L^*$. Let $u$ be an element in the kernel of $\alpha(L) : \HH^1(L, \mu')\to \HH^1(L, \Spin(Q'))$ where $\mu'$ is the center of $\Spin(Q')$. Assume $j(L)(u) = [\lambda]\in \HH^1(L, \mu_2)$ for some $\lambda\in \mathcal{O}_L^*$. Then there exist $u' \in \HH^1(\mathcal{O}_L, \mu')$ and $\epsilon'\in \mathbb{Z}$ such that $u=u'\brac{i(L)[\theta^{\epsilon'}]}$. 
\label{shapeofuandu'} \end{lemma}

\begin{proof}
Let $Z'$ denote the discriminant of $Q'$. Since $Q'$ is unramified over $L$, $Z'$ is an unramified (possibly split) quadratic extension of $L$. Thus $\theta$ is still a parameter of $Z'$. 

Suppose first that $\dim Q' \cong 2 \mod 4$ and $Z'$ is a field. Since $\HH^1(L,\mu') = \frac{U(L)}{U_{0}(L)}$, there exist $\epsilon \in \mathbb{Z}$, $f\in \mathcal{O}_L^*$ and $z\in \mathcal{O}_{Z'}^*$ with $\N_{Z'/L}(z)=f^4$ such that $u= [f\theta^{\epsilon}, z\theta^{2\epsilon}]$. Define $u' := [f, z]$ in $\HH^1(\mathcal{O}_L, \mu')$. Since $i(L)(\theta) = [\theta, \theta^2]$, it is clear that $u = u' \brac{i(L)[\theta^{\epsilon}]}$.


The case when $\dim Q' \cong 2 \mod 4$ but $Z'\simeq L\times L$ is a little more delicate. Again, clearly there exist $\epsilon, \epsilon_1 \in \mathbb{Z}$, $f, z_1, z_2 \in \mathcal{O}_L^*$  with $z_1z_2=f^4$ such that $u= [f\theta^{\epsilon}, z_1\theta^{\epsilon_1}, z_2\theta^{4\epsilon-\epsilon_1}]$. Define $u' := [f, z_1, z_2]$ in $\HH^1(\mathcal{O}_L, \mu')$.

Since $j(L)(u)=[\lambda]$ and up to squares $\lambda \in \mathcal{O}_L^*$, we see that $\epsilon_1$ introduced above is even\footnote{$[f^{-2}z_1\theta^{\epsilon_1-2\epsilon},f^{-2}z_2\theta^{2\epsilon-\epsilon_1}] = [f^{-2}z_1\theta^{\epsilon_1-2\epsilon}, 1 ]\psi[f^{-2}z_1\theta^{\epsilon_1-2\epsilon}, 1 ]^{-1}$ and $[\lambda]=[f^{-2}z_1\theta^{\epsilon_1-2\epsilon}]$.}. Since $[ab, a^4, b^4]=[1]\in \HH^1(L, \mu')$ and $i(L)(\theta) = [\theta, \theta^2, \theta^2]$, we see that

\[u = \begin{cases}
  u' & \text{if } \epsilon_1\cong 0\mod 4, \\
  u' \brac{i(L)[\theta]}  & \text{if } \epsilon_1\cong 2\mod 4.
\end{cases}\]

Suppose now that $\dim Q' \cong 0 \mod 4$. Since $\HH^1(L,\mu') = Z'^*/Z'^{*2}$, there exist $\epsilon, \epsilon_1 \in \mathbb{Z}$ and $z\in \mathcal{O}_{Z'}^*$ (resp.  $z_1, z_2 \in \mathcal{O}_{L}^*$) such that $u= [z\theta^{\epsilon}]$ (resp. $u= [z_1\theta^{\epsilon}, z_2\theta^{\epsilon_1}]$) if $Z'$ is a field (resp. a split extension). Define $u' := [z]$ (resp. $[z_1, z_2]$) in $\HH^1(\mathcal{O}_L, \mu')$. As $\lambda\in \mathcal{O}_L^*$ up to squares, we can assume\footnote{$[\lambda]=[z_1z_2\theta^{\epsilon+\epsilon_1}]$ when $Z'\simeq L\times L$.} that  $\epsilon = \epsilon_1$. Thus $u = u' \brac{i(L)[\theta^{\epsilon}]}\in \HH^1(L, \mu')$. \end{proof}

The following lemma shows that $u'$ defined above still lives in the kernel of $\alpha(L)$.

\begin{lemma}
\label{lemmaunramifiedu} Let $u, u'$ be as in Lemma \ref{shapeofuandu'}. Then $u'\in \kernel\brac{\alpha(L)}$. 
\end{lemma}

\begin{proof}  We have $u = u' \brac{i(L)[\theta^{\epsilon'}]}\in \HH^1(L, \mu')$. 

Since $Q'$ is unramified over $L$, we can view $Q'$ as the generic fiber of a diagonal quadratic form $Q'_{\mathcal{O}_L}$. Further $\Spin(Q'_L)$ is the generic fiber of a smooth reductive group scheme $\mathcal{G} = \Spin(Q'_{\mathcal{O}_L})$ over $\mathcal{O}_L$. 

Set $E := L(\sqrt{-\theta})$ if $\dim Q' \cong 2 \mod 4$ and $E:=L(\sqrt{\theta})$ if $\dim Q' \cong 0 \mod 4$. In either case, $E$ is a totally ramified extension of $L$ with residue field $\ell$. Since $i(E)(\theta)=[1]$, it is clear that $u'$ is the image of $u$ in $\HH^1(E, \mu')$. Since $\alpha(L)(u)=1$, we have $\alpha(E)(u')=1$.

\[\begin{tikzcd}
\red{u\in }\HH^1(L, \mu') \arrow[r, "\alpha(L)"] \arrow{d} & \red{1\in }\HH^1(L, \Spin(Q')) \arrow{d} \\
\red{u'\in }\HH^1(E, \mu') \arrow[r, "\alpha(E)"]  & \red{1\in} \HH^1(E, \Spin(Q'))
\end{tikzcd}\]

Since $u'$ is defined over $\mathcal{O}_L$ and hence $\mathcal{O}_E$ and the kernel of the natural map $\HH^1(\mathcal{O}_E, \mathcal{G})\to \HH^1(E, \Spin(Q'))$ is trivial (\cite{Ni84}), we see that $u'\in \kernel\brac{\alpha(\mathcal{O}_E) :  \HH^1(\mathcal{O}_E, \mu')\to \HH^1(\mathcal{O}_E, \mathcal{G})}$.

\[\begin{tikzcd}
\red{u'\in }\HH^1(\mathcal{O}_E, \mu') \arrow[r, "\alpha(\mathcal{O}_E)"] \arrow{d} & \HH^1(\mathcal{O}_E, \mathcal{G}) \arrow{d} \\
\red{u'\in }\HH^1(E, \mu') \arrow[r, "\alpha(E)"]  & \red{1\in} \HH^1(E, \Spin(Q'))
\end{tikzcd}\]

Specializing to the residue field of $E$,  we see that $\overline{u'}\in \kernel\brac{\alpha(\ell) :  \HH^1(\ell, \mu')\to \HH^1(\ell, \overline{\mathcal{G}})}$. By Hensel's lemma, this implies $u'\in \kernel\brac{\alpha(\mathcal{O}_L) :  \HH^1(\mathcal{O}_L, \mu')\to \HH^1(\mathcal{O}_L, {\mathcal{G}})}$ which shows that $\alpha(L)(u')=1$. \end{proof}

\subsection{Proof of the Theorem \ref{maintheorem}}

We proceed by induction on $\dim Q$. For the base case when $\dim Q = 2$, the norm principle holds for $\Omega(Q)$ because it is a commutative group. Assume that the norm principle holds for $\Omega(\tilde{Q})$ over $L/K$ for all even dimensional quadratic forms $\tilde{Q}/K$ with $\dim \tilde{Q} < \dim Q$. We break up the proof into separate cases depending on the ramification of $L/K$. 

\subsection*{Case I : $L/K$ is unramified}
Recall that we have reduced to the case where $\lambda\in \mathcal{O}_L^*$ up to squares. Using the second residue map again, we see that $[\overline{p}_{\ell}] = \delta_{2,t}(Q_L)=\delta_{2, t}\brac{\lambda Q_L} = [\overline{\lambda} \overline{p}_{\ell}] \in W(\ell)$. Similarly $\overline{q}_{\ell}\simeq \overline{\lambda} \overline{q}_{\ell}$. Thus $\lambda \in \G(q_L)\cap \G(p_L)$. 

Note that if $\dim q$ is odd, since $[{\langle 1, -\lambda \rangle}_L \otimes q_L] = 0 \in W(L)$, then by (\cite{Sch85}, Theorem 10.13), we have $\lambda\in L^{*2}$, i.e $[\lambda]=1\in \HH^1(L, \mu_2)$. Thus by Lemma \ref{lemmaspinornorm}, the norm principle holds.

So we assume that $\dim q$ is even. Therefore $\dim p$ is even too. Set $f_u = q$, $g_u = tp$, $R_u := \Oo^+(f_u)\times \Oo^+(g_u)\subset \Oo^+(Q)$ and $\tilde{R}_u$, the preimage of $R_u$ under the canonical homomorphism $\Spin(Q)\to \Oo^+(Q)$. By Theorem \ref{theoreminductiveapproach}, it suffices to show that the norm principle holds for $\Omega(f_u)$ and $\Omega(g_u)$ over $L/K$, which holds by induction if $\dim q, \dim p\neq 0$.

Without loss of generality\footnote{If $\dim q=0$, $Q$ is similar to the unramified form $p$ over $K$ and $\Omega(Q)\simeq \Omega(p)$. The same proof works in this case.}, suppose that $\dim p=0$ and $Q_K \simeq q_K$. Since $Q$ is unramified over $K$ and hence $Q_L$ is unramified over $L$, using Lemmata \ref{shapeofuandu'} and \ref{lemmaunramifiedu}, we find $u'\in \HH^1(\mathcal{O}_L, \mu)$  such that $u=u'\brac{i(L)[t^{\epsilon'}]}$ for some $\epsilon'\in \mathbb{Z}$ and $\alpha(L)(u')=1$. The proof of Lemma \ref{lemmaunramifiedu} in fact shows that we can specialize to the residue field and get $\overline{u'}$ in the kernel of $\alpha(\ell): \HH^1(\ell, \mu) \to \HH^1(\ell, \Spin(\overline{Q_L}))$. Since the norm principle holds for $\alpha$ for the quadratic form $\overline{Q}$ over $\ell/k$ by assumption and Lemma \ref{lemmaequivnormprinciple}, $\N_{\ell/k}(\overline{u'})$ is in the kernel of $\alpha(k): \HH^1(k, \mu) \to \HH^1(k, \Spin(\overline{Q})$. Thus by Hensel's Lemma, $\alpha(K)\brac{\N_{L/K}(u')}=1$. Lemma \ref{lemmaspinornorm} implies that $\N_{L/K}(i(L)[t^{\epsilon'}])\in \kernel\brac{\alpha(K)}$. Thus $\N_{L/K}(u)\in \kernel\brac{\alpha(K)}$.

\subsection*{Case II : $L/K$ is ramified}
Recall once again that we have reduced to the case where $\lambda\in \mathcal{O}_L^*$ up to squares. Using Hensel's lemma, we can assume in fact that $\lambda\in \mathcal{O}_K^*$. Further we can also assume $\lambda\not\in L^{*2}$ as otherwise, we would be done by Lemma \ref{lemmaspinornorm}.

\textbf{Subcase IIa}: We first look at the situation when $\dim p \neq \dim q$. Then the following Lemma in conjunction with Theorem \ref{theoreminductiveapproach} and our induction hypothesis finishes the proof in this case.

\begin{lemma}
\label{lemmalambdaspinornorm}
Suppose that $\dim p\neq \dim q$. Then $q\otimes \langle 1, -\lambda\rangle$ or $p\otimes \langle 1, -\lambda \rangle$ is isotropic over $K$. Further $Q_K \simeq  f_K\perp g_K$ for $f, g$ even dimensional quadratic forms over $K$ with $\lambda f_L \simeq f_L$ and $\lambda g_L\simeq g_L$. 
\end{lemma}
\begin{proof} Since $\lambda$ is a multiplier for the form $q\perp cp$, we have $[{\langle 1, -\lambda \rangle}_L \otimes {q \perp \langle 1, -\lambda \rangle}_L \otimes cp_L] = 0 \in W(L)$. Without loss of generality, assume $\dim q > \dim p$. Hence $q\otimes \langle 1,-\lambda\rangle$ is isotropic over $L$.

Recall that $q \simeq \langle a_1, a_2, \ldots, a_{\dim q}\rangle$ for $a_i\in \mathcal{O}_K^*$ and $\lambda\in \mathcal{O}_K^*$. Thus $q\otimes \langle 1,-\lambda \rangle \simeq \langle a'_1, a'_2, \ldots, a'_{2\dim q}\rangle$ for some $a'_i\in \mathcal{O}_K*$. Hence there exist $w_i \in \mathcal{O}_L^*\cup\{0\}$ not all $0$ and $t_{i}\in \mathbb{Z}$ such that $\sum a'_iw_i^2\theta^{2t_i} = 0$ where $\theta$ is a parameter of $L$. By cancelling factors of $\theta^{2}$ if necessary, we can assume that each $t_j\geq 0$ and at least one $t_i = 0$ (with corresponding $w_i\neq 0$). That is $\overline{q}\otimes \langle 1,- \overline{\lambda}\rangle$ is isotropic over $k$. By Hensel's lemma, $q\otimes \langle 1, -\lambda\rangle$ is isotropic over $K$.

Since we have assumed $q$ is anisotropic over $K$, this implies $q(u) = \lambda q(v)$ for $u, v\in K^{\dim q}$ and $q(u), q(v)\neq 0$. Note that if $u, v$ are linearly dependent over $K$, then $\lambda\in K^{*2}$ contradicting our assumption that $\lambda$ is not a square. Thus $u, v$ span a two dimensional $K$-vector space $W$. Let $f$ be the $K$-quadratic form $q$ restricted to $W$.

Let $q(v) = a, q(u) = \lambda a$ and $b_q(u,v)=x$. Then for orthogonal basis $\{v, u-\frac{x}{a}v\}$, $f_K \simeq {\langle a, \lambda a - \frac{x^2}{a} \rangle}_K$ and for orthogonal basis $\{u, v -\frac{x}{\lambda a}u\}$, $f_K \simeq \langle {\lambda a, a - \frac{x^2}{\lambda a} \rangle}_K$. Thus $\lambda f_L \simeq f_L$ and hence the Lemma follows. \end{proof}

\textbf{Subcase IIb}: Assume now that $\dim p = \dim q$. Since $Q_L$ is unramified over $L$, use Lemmata \ref{shapeofuandu'} and \ref{lemmaunramifiedu} to find $u'\in \HH^1(\mathcal{O}_L, \mu)$  such that $u=u'\brac{i(L)[\theta^{\epsilon'}]}$ for some $\epsilon'\in \mathbb{Z}$ and $\alpha(L)(u')=1$.

\begin{lemma}
$\N_{L/K}(u') = 1 \in \HH^1(K, \mu)$.
\end{lemma}

\begin{proof} We only give the proof in the case $Z_L$ is a field. The proof when $Z_L\simeq L\times L$ is similar.

If $\dim q = \dim p$ is odd, then the discriminant $Z$ of $Q = q\perp tp$ is a totally ramified quadratic extension of $K$ and $Z_L := Z\otimes_K L$ is an unramified (possibly split) quadratic extension of $L$. Thus $\theta$ is still a parameter of $Z_L$. Further, if the residue field of $Z_L = \ell'$, then the norm maps $\N_{Z_L/L} : Z_L^*\to L^*$ and $\N_{Z_L/Z} : Z_L^*\to Z^*$ induce the same map $\N_{\ell'/k} : {\ell'}^*\to k^*$ at the residue field level.

\[\begin{tikzcd}
L \arrow[r, "\mathrm{unram}"] & Z_L \\
K \arrow[r, "\mathrm{ram}"] \arrow[u, "\mathrm{ram}"]  & Z \arrow[u, "\mathrm{unram}"] 
\end{tikzcd}\]

Recall that $\HH^1(L,\mu) = \frac{U(L)}{U_{0}(L)}$ and $u' =  [f, z]$  for some $f\in \mathcal{O}_L^*$ and $z\in \mathcal{O}_{Z_L}^*$ with $\N_{Z_L/L}(z)=f^4$. Thus we have $\overline{\N_{Z_L/Z}(z)} = \N_{\ell'/k}(\overline{z})=\overline{f}^4$. By Hensel's Lemma, there exist $\tilde{f}\in \mathcal{O}_K^*$ and $x\cong 1 \mod (m_{\mathcal{O}_L})$ in $\mathcal{O}_L^*$ such that $f = x\tilde{f}$. Similarly, there exist $y, y'\cong 1 \mod (m_{\mathcal{O}_Z})$ in $\mathcal{O}_Z^*$ such that $\N_{Z_L/Z}(z) = \tilde{f}^4y' = \tilde{f}^4y^4$.

Now $\N_{L/K}(u') = [\N_{L/K}(f), \N_{Z_L/Z}(z)] = [\tilde{f}^2\N_{L/K}(x), \tilde{f}^4y^4] = [\tilde{f}^2, \tilde{f}^4][\N_{L/K}(x), y^4]\in \HH^1(K, \mu)$ where $\N_{Z/K}(y^4) = \N_{L/K}(x)^4$. Setting $a = \N_{Z/K}(y)\N_{L/K}(x)^{-1}$, we see that $a^4=1$ and $a\cong 1 \mod m_{\mathcal{O}_K}$. By Hensel's Lemma yet again, $a=1$ and hence $\N_{Z/K}(y)=\N_{L/K}(x)$. Since $[\N(b), b^4]=1\in \HH^1(K, \mu)$ for every $b\in K^*$, we see that $\N_{L/K}(u')=1$.

If $\dim q = \dim p$ is even, then the discriminant $Z$ of $Q = q\perp tp$ is an unramified (possibly split) quadratic extension of $K$ as also $Z_L/L := Z\otimes_K L/L$. Thus $\theta$ is still a parameter of $Z_L$. Note that $Z_L \simeq L\times L$ if and only if $Z\simeq K\times K$. 

\[\begin{tikzcd}
L \arrow[r, "\mathrm{unram}"] & Z_L \\
K \arrow[r, "\mathrm{unram}"] \arrow[u, "\mathrm{ram}"]  & Z \arrow[u, "\mathrm{ram}"] 
\end{tikzcd}\]

Recall that $\HH^1(L,\mu) = Z_L^*/Z_L^{*2}$ and $u' = [z]$ in $\HH^1(L, \mu)$ for some $z\in \mathcal{O}_{Z_L}^*$ if $Z_L$ is a field (resp. a split extension). Since norms of units of totally ramified quadratic extensions are squares and $Z_L/Z$ is ramified, $\N_{L/K}(u') =  1 \in \HH^1(K, \mu)$.

\end{proof}

Lemma \ref{lemmaspinornorm} implies that $\N_{L/K}(i(L)[\theta^{\epsilon'}])\in \kernel\brac{\alpha(K)}$. Since $\N_{L/K}(u)=\N_{L/K}(u')\N_{L/K}\brac{\brac{i(L)[\theta^{\epsilon'}]}}$, we have $\N_{L/K}(u)\in \kernel\brac{\alpha(K)}$ which concludes the proof of Theorem \ref{maintheorem}.

\section{Examples}
Let $G$ be a semisimple simply connected linear algebraic group defined over a field $k$. Then $\HH^1(k, G)$ is trivial if $k$ is a $p$-adic field (\cite{K65}) or a global field of positive characteristic (\cite{H75}).  More generally, suppose that $\charac(k)\neq 2$ and $\cd(k)\leq 2$. Then Bayer-Parimala's proof of Serre's conjecture II shows that $\HH^1(k, G)$ is trivial if $G$ is further assumed to be of classical type (\cite{BP95}). Set $G :=\Spin(q)$ where $q$ is any even dimensional nondegenerate quadratic form over $k$ and let $\mu$ be its center. It is immediate therefore that the norm principle holds for the map $\alpha: \HH^1(-,\mu)\to \HH^1(-, \Spin(q))$ and hence that it holds for the group $\Omega(q)$ defined over $k$.

Now let $G$ be a semisimple simple adjoint linear algebraic group of classical type defined over a number field $k$. Then, $G(k)/R$,  the group of $R$-equivalence classes of the $k$-points of $G$, is trivial (\cite{Gi97} Corollaire III.4.2 and \cite{KP08}, pg 1). Set $G :=\PGO^+(q)$ where $q$ is any even dimensional nondegenerate quadratic form over $k$ and let $\mu$ be the center of $\Spin(q)$. It follows from Theorem \ref{gillenormprinciple} that the norm principle holds for the map $S: \PGO^+(q)(-)\to \HH^1(-, \mu)$ and hence that it holds for the group $\Omega(q)$ defined over $k$.

Recall that a field $k$ is said to have \textit{virtual cohomological dimension} ($\vcd$) $\leq n$ if the cohomological dimension of $k\brac{\sqrt{-1}}$ is $\leq n$. Examples of $\vcd \leq 2$ fields include $\cd \leq 2$ fields and number fields. We begin by showing the following:

\begin{lemma}
\label{lemmaRtrivialPGOoverrealclosed} Let $F$ be a real closed field and $q'$, a non-degenerate even dimensional quadratic form over $F$. Then $\PGO^+(q')(F)/R$ is trivial.
\end{lemma}
\begin{proof}
Without loss of generality we can assume that $q'$ is anisotropic over $F$. Since for every $a\in F$, either $a$ or $-a$ is a square in $F$, we can further assume that  $q' \simeq \langle 1,\ldots,1\rangle$. Then the variety of $\PGO^+(q')$ is stably rational
over $F$ (\cite{Ch94}), whence the claim. \end{proof}

\begin{proposition}
\label{propnormprinciplevcd2}
Let $k$ be a field with $\vcd(k)\leq 2$ and $q$, a non-degenerate even dimensional quadratic form over $k$. Then the norm principle holds for $\Omega(q)$.
\end{proposition}

\begin{proof} If $\cd(k)\leq 2$, the discussion above already gives the proof. Hence we can assume that $\cd(k)\neq \vcd(k)$ and hence that $\charac(k)=0$ and $k$ can be ordered (Theorem 1.1, \cite{BP98}).

Let $\ell/k$ be a finite separable extension, $\mu$, the center of $\Spin(q)$, and let $\alpha(-) : \HH^1(-,\mu)\to \HH^1(-, \Spin(q))$ be the natural map between the Galois cohomology sets. Let $\xi\in \kernel \brac{\alpha(\ell)}$ and  set $\eta:=N_{\ell/k}(\xi)$. We would like to show that $\alpha(k)\brac{\eta}$ is trivial in $\HH^1(k, \Spin(q))$.

Let $\Omega$ denote the set of all orderings $v$ of $k$ and $k_v$, the real closure of $k$ at $v$. Then the Hasse principle result of Bayer-Parimala over perfect fields of $\vcd \leq 2$ for semisimple simply connected groups of classical type (\cite{BP98}) gives in particular that the natural map $\HH^1(k, \Spin(q)) \to\prod_{v\in \Omega} \HH^1(k_v,\Spin(q))$ has trivial kernel. 

Thus, it suffices to show that the image of $\eta$ in $H^1(k_v,\Spin(q))$ is trivial for each $v\in \Omega$. Note that $\mathrm{Res}_{k_v}(\eta)=N_{\ell\otimes_k k_v/k_v}(\xi)$. By Lemma \ref{lemmaRtrivialPGOoverrealclosed}, $\PGO^+(q)(\ell\otimes_k k_v)/R$ is trivial and hence as the norm principle holds for $R$-trivial elements (Theorem \ref{gillenormprinciple}), we can conclude that the image of ${\rm Res}_{k_v}(\eta)$ in $H^1(k_v,\Spin(q))$ is trivial. \end{proof}

Since virtual cohomological dimension behaves well with respect to finite extensions, Proposition \ref{propnormprinciplevcd2} in conjunction with Theorem \ref{maintheorem} immediately yields the following:

\begin{corollary}
\label{corollary-examples} Let $K$ be a complete discretely valued field with residue field $k$ such that $\charac(k)\neq 2$ and $\vcd(k)\leq 2$. Then the norm principle holds for $\Omega(Q)$ for every even dimensional non-degenerate quadratic form $Q/K$.
\end{corollary}



\section{On the triviality of the Tate-Shaferevich set}
\label{tateshaferevichsection}
Let $G$ be a semisimple algebraic group over a number field $K$ and $V^K$, the set of all places of $K$. One of the main finiteness results in the arithmetic theory of linear algebraic groups states that the natural global-to-local map 
\[\rho_{G} : \HH^1(K , G) \to \prod_{v \in V^K} \HH^1(K_v , G)\]
is \textit{proper}, i.e. the preimage of any finite set is finite; in particular, the corresponding Tate-Shafarevich set {\brus SH}$(G) := \kernel\: \rho_G$ is finite. Moreover, if in addition $G$ is simply connected, then   {\brus SH}$(G)=1$, i.e.
$\rho_G$ is injective.

A natural question to ask is if, and to what extent, the above finiteness property can be extended to fields other than number fields. More precisely, let $K$ be a finitely generated field. Can one equip $K$ with a ``natural" set $V$ of discrete valuations such that for a given absolutely almost simple  $K$-group $G$, the natural global-to-local map relative to $V$
\[\rho_{G, V} : \HH^1(K , G) \to \prod_{v \in V} \HH^1(K_v , G)\]
 is proper ? If the answer is affirmative, is it true that for a simply connected group $G$
 the kernel {\brus SH}${}_V(G)$ of  $\rho_{G, V}$  is trivial ?

A natural candidate for such a $V$ appears to be the set of discrete valuations associated to the prime divisors of a model of $K$, i.e. a smooth affine arithmetic scheme with function field $K$
(we call such sets {\it divisorial}).
It is known that divisorial sets $V$ indeed work for adjoint inner forms of type $\textsf{A}_{\ell}$, i.e. for $G={\rm PGL}_{\ell+1}$,
 provided that $\mathrm{char}\: K$ does not divide $\ell + 1$ (cf. \cite{CRR13}); this relies on the finiteness of the unramified Brauer group ${}_{(\ell + 1)}\Br(K)_V$ (\cite{CRR16}). Until recently no other types have been considered.

\subsubsection*{Relating the norm principle and the Tate-Shaferevich set of spinor groups}
\phantom{.}

The first interesting and  widely open case is the one when $K$ is the function field of a curve
$C$ defined over a number field and $G=\Spin(f)$ is the spinor group of a quadratic form $f$ over $K$.
In a recent paper (\cite{CRR16(2)})  it was proved that for the special orthogonal
group $\Oo^+(f)$  and for a divisorial set $V$, the global-to-local map $\rho_{\Oo^+(f),V}$ is proper.  This result in conjunction with the exact sequence \[\Oo^+(f)(K) \longrightarrow \HH^1(K,\mu_2) \longrightarrow \HH^1(K,\Spin(f)) \longrightarrow \HH^1(K,{\Oo^+}(f)),\] and twisting shows that the Tate-Shafarevich set {\brus SH}$_V(\Spin(f))$ is finite if and only if
the group $\mathrm{LGC}(g)=\{[a]\in K^\times/\Sn(g) \ | \ a \mbox{   is a spinor norm of } g \mbox{ over   } K_v \mbox{ for all  } v\in V \}$ is finite for all quadratic forms $g$ over $K$. Note that $\mathrm{LGC}(g)$ is a subset of {\brus SH}$_V(\Spin(g))$.

Note that the residue field $\kappa(v)$ of any valuation $v\in V$ is either a number field or the function field of a curve over a finite field, i.e., $\vcd(\kappa(v))\leq 2$. Theorefore by Corollary \ref{corollary-examples}, the norm principle for $\Omega(f)$ holds over the completion $K_v$. Thus by Theorem \ref{thmnormprinciplesquare}, the square diagram (SQ for $X/K_v$) is commutative for every finite separable extension $X/K_v$. Let $L/K$ be a finite separable extension. It follows then that the obstruction to the commutativity of the square diagram (SQ for L/K) is a subgroup in $\mathrm{LGC}(f)$ having the property: \textit{it is trivial if and only if the norm principle holds for the extension $L/K$}. Thus, the failure  of the norm principle for $\Spin(f)$ would imply that {\brus SH}$_V(\Spin(f))$ is nontrivial.


\begin{thebibliography}{9999}
\bibitem[BM00]{BM00} P. Barquero and A. Merkurjev, Norm Principle for Reductive Algebraic Groups, \textit{J. Proceedings of the International Colloquium on Algebra, Arithmetic and Geometry TIFR, Mumbai} (2000).
\bibitem[BP95]{BP95} E. Bayer-Fluckiger and R. Parimala, Galois cohomology of the classical groups over fields of cohomological dimension $\leq$ 2, \textit{Invent. Math.} \textbf{2} (1995) : pgs 195-229.
\bibitem[BP98]{BP98} E. Bayer-Fluckiger and R. Parimala, Classical groups and the Hasse principle, \textit{Annals of mathematics} \textbf{147(3)} (1998) : pgs 651-693.
\bibitem[Bh16]{Bh16} N. Bhaskhar,  On Serre's injectivity question and norm principle, \textit{Commentarii Mathematici Helvetici} \textbf{91(1)} (2016) : pgs 145-161.
\bibitem[Ch94]{Ch94} V. Chernousov,  The group of multipliers of the canonical quadratic form
and stable rationality of the variety ${\rm PSO}$, \textit{Matematicheskie zametki} {\bf 55} (1994) : pgs 114-119.

\bibitem[CRR13]{CRR13} V.I.~Chernousov, A.S.~Rapinchuk and I.A.~Rapinchuk, The genus of a division algebra and the unramified Brauer group, \textit{Bull. Math. Sci.} {\bf 3} (2013): pgs 211-240.



\bibitem[CRR16]{CRR16}  V.I.~Chernousov, A.S.~Rapinchuk and I.A.~Rapinchuk, On the size of the genus of a division algebra, \textit{Proc. Steklov Inst. of Math.} {\bf 292(1)} (2016) : pgs 63-93.

\bibitem[CRR16(2)]{CRR16(2)} V.I.~Chernousov, A.S.~Rapinchuk and I.A.~Rapinchuk,  On some finiteness properties of
algebraic groups over finitely generated fields, \textit{C. R.  Acad. Sci. Paris}, Ser. I,  {\bf 354} (2016) : pgs 869-873.

\bibitem[Gi93]{Gi93} P. Gille, R-\'equivalence et principe de norme en cohomologie galoisienne, \textit{Comptes rendus de l'Acad\'emie des sciences. S\'erie 1, Mathématique} \textbf{316(4)} (1993) : pgs 315-320.
\bibitem[Gi97]{Gi97} P. Gille, La R-\'equivalence sur les groupes alg\'ebriques r\'eductifs, \textit{Publications math\'ematiques de l’IH\'ES} \textbf{86} (1997) : pgs 199-235.
\bibitem[H75]{H75} G. Harder, \"Uber die Galoiskohomologie halbeinfacher algebraischer Gruppen. III. \textit{Journal f\"ur die reine und angewandte Mathematik} \textbf{274} (1975) : pgs 125-138.
\bibitem[K65]{K65} M. Kneser, Galois-Kohomologie halbeinfacher algebraischer Gruppen \"uber $p$-adischen K\"orpern. I, \textit{Math. Z. } \textbf{88} (1965) : pgs 40-47, II \textit{ibid} \textbf{89} (1965) : pgs 250-272.
\bibitem[KMRT98]{KMRT98} M.-A. Knus, A. Merkurjev, M. Rost, and J.-P. Tignol, The book of involutions, American Mathematical Society Colloquium Publications \textbf{44} (1998), AMS.
\bibitem[KP08]{KP08} A. Kulshrestha and R. Parimala, R-equivalence in adjoint classical groups over fields of virtual cohomological dimension 2, \textit{Transactions of the American Mathematical Society} \textbf{360(3)} (2008) : pgs 1193-1221.
\bibitem[Me96]{Me96} A. Merkurjev,  A norm principle for algebraic groups, \textit{St Petersburg Math} \textbf{7 (2)} (1996) : pgs 243-264.
\bibitem[Me96(2)]{Me96(2)} A. Merkurjev, $R$-equivalence and rationality problem for semisimple adjoint classical groups, \textit{Publications Math\'ematiques de l'IH\'ES} \textbf{84} (1996) : pgs 189-213.
\bibitem[Ni84]{Ni84} Y. Nisnevich, Rationally Trivial Principal Homogeneous Spaces and Arithmetic of Reductive Group Schemes Over Dedekind Rings, \textit{C. R. Acad. Sci. Paris}, S\'erie I, \textbf{299} , no. 1 (1984) : pgs 5–8.
\textit{Ast\'erisque} \textbf{227 (783(4))} (1995) : pgs 229-257.
\bibitem[Sch85]{Sch85} W. Scharlau, Quadratic and Hermitian forms, Springer-Verlag (1985).




\end{thebibliography}
\end{document}